\documentclass[10pt,a4paper]{article}

\usepackage{latexsym,amssymb}

\begin{document}
\title{\bf When is the Direct Product of Generalized
Mycielskians a Cover Graph?}

\author{Hsin-Hao Lai\\
\normalsize Department of Mathematics \\
\normalsize National Kaohsiung Normal University\\
\normalsize Yanchao, Kaohsiung 824, Taiwan\\
\normalsize {\tt Email: hsinhaolai@nknucc.nknu.edu.tw}\\
\and
Ko-Wei Lih\\
\normalsize Institute of Mathematics\\
\normalsize Academia Sinica\\
\normalsize Nankang, Taipei 115, Taiwan\\
\normalsize {\tt Email: makwlih@sinica.edu.tw}\\
\and
Chen-Ying Lin\thanks{The corresponding author. Supported in part by the
        National Science Council under grant NSC96-2115-M-366-001.}\\
\normalsize Department of Computer Science\\
\normalsize and Information Engineering\\
\normalsize Shu-Te University \\
\normalsize Kaohsiung 824, Taiwan\\
\normalsize {\tt Email: lincy@mail.stu.edu.tw}\\
\and
Li-Da Tong\thanks{Supported in part by the
        National Science Council under grant NSC95-2115-M-110-012-MY2.}\\
\normalsize Department of Applied Mathematics\\
\normalsize National Sun Yat-sen University\\
\normalsize Kaohsiung 804, Taiwan\\
\normalsize {\tt Email: ldtong@math.nsysu.edu.tw}
}

\date{}
\maketitle

\newcommand{\cqfd}{\hfill \rule{8pt}{9pt}}
\newcommand{\la}{\langle}
\newcommand{\ra}{\rangle}
\newtheorem{define}{Definition}
\newtheorem{proposition}[define]{Proposition}
\newtheorem{theorem}[define]{Theorem}
\newtheorem{lemma}[define]{Lemma}
\newtheorem{remark}[define]{Remark}
\newtheorem{corollary}[define]{Corollary}
\newtheorem{problem}[define]{Problem}
\newtheorem{conjecture}[define]{Conjecture}
%
%
\newenvironment{proof}{
\par
\noindent {\bf Proof.}\rm}%
{\mbox{}\hfill\rule{0.5em}{0.809em}\par}
%
\newfont{\Bb}{msbm10 scaled\magstep1}
\baselineskip=16pt
\parindent=0.5cm

\bigskip


\begin{abstract}
\noindent
A graph is said to be a cover graph if it is the underlying
graph of the Hasse diagram of a finite partially ordered set.
The direct product $G \times H$ of graphs $G$ and $H$ is the
graph having vertex set $V(G) \times V(H)$ and edge set $E(G
\times H) = \{(g_i,h_s)(g_j,h_t)\mid g_ig_j \in E(G) \mbox{ and }
h_sh_t \in E(H)\}$. We prove that the direct product ${\sf M}_m(G)
\times {\sf M}_n(H)$ of the generalized Mycielskians of $G$ and $H$
is a cover graph if and only if $G$ or $H$ is bipartite.


\end{abstract}

%
\section{Introduction}
%

All graphs are assumed to be finite and simple in this paper.
Let $uv$ denote the edge joining the vertices $u$ and $v$.
The {\em girth} $g(G)$ of a graph $G$ is the length of a shortest
cycle in $G$ if there is any, and $\infty$ if $G$ possesses no
cycles. The {\em chromatic number} $\chi(G)$ is the least
number of colors to be assigned to the vertices of $G$ so that
adjacent vertices receive distinct colors. The Hasse diagram of
a finite partially ordered set depicts the covering relation of
elements; its underlying graph is called a {\em cover graph}.
The problem of characterizing cover graphs goes back to Ore
\cite{ore}. It is an NP-complete problem to decide whether a
graph $G$ is a cover graph (see \cite{birightwell} and \cite{rodl-2}).

Let the graph $G$ be endowed with an acyclic orientation $D$, i.e.,
there are no directed cycles with respect to $D$.
An arc of $D$ is called {\em dependent} if its reversal creates
a directed cycle in $D$. Let $d_{\min}(G)$ be the minimum number
of dependent arcs over all acyclic orientations of $G$.
Pretzel \cite{pret} proved that $d_{\min}(G)=0$ is equivalent to
$G$ being a cover graph. It follows immediately that a triangle
is not a cover graph. The following is a well-known sufficient
condition to identify a cover graph. It was obtained in Aigner
and Prins \cite{ag} and first appeared in Pretzel and Youngs
\cite{balanced}. A simple proof is included in Fisher et al.
\cite{fisher}.

\begin{theorem}\label{chig}
A graph $G$ is a cover graph if $\chi(G) < g(G)$.
\end{theorem}

A {\em homomorphism} from a graph $G$ to a graph $H$ is a
mapping from $V(G)$ to $V(H)$ such that $f(u)f(v) \in E(H)$
whenever $uv \in E(G)$. We denote $G \rightarrow H$ if there
is a homomorphism from $G$ to $H$. Note that $G_1 \rightarrow
G_2$ and $G_2 \rightarrow G_3$ imply $G_1 \rightarrow G_3$
for graphs $G_1, G_2$, and $G_3$. Lih et al. \cite{tower}
proved the following.

\begin{lemma}\label{2homoto1}
Assume that $G \rightarrow H$.
If $H$ is a cover graph, then $G$ is a cover graph.
\end{lemma}

The {\em direct product} $G \times H$ of graphs $G$ and
$H$ is the graph with vertex set $V(G) \times V(H)$ and edge
set $E(G \times H) = \{(g_i,h_s)(g_j,h_t)\mid g_ig_j \in E(G)
\mbox{ and } h_sh_t \in E(H)\}$. Other names for the direct
product include tensor product, categorical product,
Kronecker product, cardinal product, relational product,
weak direct product, etc.
The following is an outstanding conjecture of Hedetniemi
\cite{het} involving the direct product.

\begin{conjecture}
$\chi(G \times H) = \min \{\chi(G), \chi(H)\}.$
\end{conjecture}

The reader is referred to Imrich
and Klav\v zar \cite{ik} for more information on direct products.
It is easy to see that $G \times H \rightarrow G$ and
$G \times H \rightarrow H$. The following theorem is a
consequence of Lemma \ref{2homoto1}.

\begin{theorem}\label{dir}
The direct product $G \times H$ is a cover graph
if $G$ or $H$ is a cover graph.
\end{theorem}

Since $G$ is isomorphic to the subgraph of $G \times G$
induced by the set $\{(v, v)\mid v\in V(G)\}$, we have
the following consequence.

\begin{corollary}\label{G*G}
The direct product $G \times G$ is not a cover graph
if $G$ is not a cover graph.
\end{corollary}

\begin{theorem}\label{thm5}
We have the following two cases.
\begin{enumerate}
\item  Suppose that $g(G)=g(H)=3$.
Then $G \times H$ is not a cover graph.

\item  Suppose that $g(G)>3 \geqslant \chi(H)$.
Then $G \times H$ is a cover graph.
\end{enumerate}
\end{theorem}

\begin{proof}
If $g(G)=g(H)=3$, then $G\times H$ contains a triangle;
that is, $G\times H$ is not a cover graph.
If $g(G)>3 \geqslant \chi(H)$, then we have that $g(G
\times H)>3 \geqslant \chi(H)$.
Since $G \times H \rightarrow H$, $\chi(H) \geqslant
\chi(G \times H)$.
By Theorem \ref{chig}, $G \times H$ is a cover graph.
\end{proof}

\bigskip

Probabilistic arguments (\cite{bbn} for example) show that it is
common to have triangle-free graphs that are not cover graphs.
Take the direct product of such a graph with a triangle. Theorem
\ref{thm5} leaves open possibility that this direct product could
be a cover graph even if each of the two factors is not a
cover graph. In the next section we shall show that some direct
product is not a cover graph when each of the two components has
girth 4, chromatic number 4, and is not a cover graph. We shall
also answer the question posed in the title.
A general problem remains to be completely solved is to determine
whether the product $G \times H$ is not a cover graph while both
$G$ and $H$ are not cover graphs.

\bigskip

%
\section{Direct product of generalized Mycielskians}
%

Let $G=(V_0, E_0)$ be a graph with vertex set $V_0 = \{\la 0,0
\ra, \la 0,1 \ra, \ldots , \la 0,n-1 \ra\}$ and edge set $E_0$.
For $m>0$, the {\em generalized Mycielskian} ${\sf M}_m(G)$ of
$G$ has vertex set $V=V_0 \cup (\cup_{i=1}^{m} V_i) \cup \{u\}$,
where $V_i=\{\la i,j \ra \mid 0 \leqslant j \leqslant n-1\}$ for
$1 \leqslant i \leqslant m$, and edge set $E=E_0 \cup
(\cup_{i=1}^{m} E_i) \cup \{\la m,j \ra u \mid 0 \leqslant j
\leqslant n-1\}$, where $E_i= \{\la i-1,j \ra \la i,k \ra \mid
\la 0,j \ra \la 0,k \ra \in E_0\}$ for $1 \leqslant i \leqslant
m$. We call $u$ the root in ${\sf M}_m(G)$. In the sequel, the
second coordinates of vertices in ${\sf M}_m(G)$ are always taken
modulo the order $n$. We note that ${\sf M}_1(G)$ is commonly
known as the {\em Mycielskian} of $G$. It is easy to see that if
$H$ is a subgraph of $G$, then ${\sf M}_m(H)$ is a subgraph of
${\sf M}_m(G)$. The following theorem appeared in Lih et al.
\cite{tower}.

\begin{theorem}\label{iffeven}
Let $n \geqslant 3$. The graph ${\sf M}_m(C_n)$
is a cover graph if and only if $n$ is even.
\end{theorem}

This can be further generalized into the following.

\begin{theorem}\label{lai}
The graph ${\sf M}_m(G)$ is a cover graph if and only
if $G$ is bipartite.
\end{theorem}

\begin{proof}
If $G$ has no edge, then obviously ${\sf M}_m(G)$ is a cover
graph. Let $G$ be a bipartite graph with at least one edge. Then
$\chi({\sf M}_m(G)) = 3 < 4 \leqslant g({\sf M}_m(G))$. Hence
${\sf M}_m(G)$ is a cover graph. If $G$ is not bipartite, then
$G$ contains an odd cycle $C$ with length at least 3. By Theorem
\ref{iffeven}, ${\sf M}_m(C)$ is not a cover graph. Hence, ${\sf
M}_m(G)$, being a supergraph of ${\sf M}_m(C)$, is not a cover
graph.
\end{proof}

\bigskip

The following two lemmas can be verified in a straightforward
manner.

\begin{lemma}\label{product-of-subgraph}
Let $H_i$ be a subgraph of $G_i$ for $i=1$ and $2$.
Then $H_1 \times H_2$ is a subgraph of $G_1 \times G_2$.
\end{lemma}

\begin{lemma}\label{homo-to-subgraph}
Let $H_i \rightarrow G_i$ for $i=1$ and $2$.
Then $H_1 \times H_2 \rightarrow G_1 \times G_2$.
\end{lemma}

The next lemma will be used in proving our main results.

\begin{lemma}\label{big-homo-to-small}
Let $p \geqslant m \geqslant 1$ and $q \geqslant s \geqslant 1$.
Then we have ${\sf M}_p(C_{2q+1}) \rightarrow {\sf M}_m(C_{2s+1})$.
\end{lemma}

\begin{proof}
Let the cycle $C_n$ with $n$ vertices be denoted
$\la 0,0 \ra \la 0,1 \ra \cdots  \la 0,n-1 \ra \la 0,0 \ra$.
The required homomorphism will be constructed in two stages.

\medskip

{\bf Stage 1}.\
Define a mapping $\sigma: V({\sf M}_p(C_{2q+1})) \rightarrow
V({\sf M}_m(C_{2q+1}))$ as follows.
Let $\sigma(u)=u$. We always assume that $0 \leqslant j \leqslant 2q$.
Now let
$$
\sigma(\la i, j \ra )=  \left \{
\begin{array} {ll}
\la 0, j \ra     & \mbox{ if } 0     \leqslant i \leqslant p-m, \\
\la i-p+m, j \ra & \mbox{ if } p-m+1 \leqslant i \leqslant p. \\
\end{array}
\right.
$$

To check that $\sigma$ is a homomorphism from ${\sf M}_p(C_{2q+1})$
to ${\sf M}_m(C_{2q+1})$, let $xy$ be an edge of ${\sf M}_p(C_{2q+1})$.

If $x=\la 0,j \ra$ and $y=\la 0,j+ 1 \ra $,
then $\sigma(x)\sigma(y)=\la 0,j \ra \la 0,j+1 \ra$.

If $x=\la i,j \ra $, $y=\la i-1,j \pm 1 \ra $, and $1 \leqslant
i \leqslant p-m$, then
$\sigma(x)\sigma(y)=\la 0,j \ra \la 0,j \pm 1 \ra$.

If $x=\la i,j \ra $, $y=\la i-1,j \pm 1 \ra $, and $p-m+1 \leqslant i
\leqslant p$, then $\sigma(x)\sigma(y)=\la
i-p+m,j \ra \la i-p+m-1,j \pm 1 \ra$.

If $x=\la p,j \ra$ and $y= u$,
then $\sigma(x)\sigma(y)=\la m,j \ra u$.

We see that $\sigma(x)\sigma(y) \in E({\sf M}_m(C_{2q+1}))$ in all cases.
Hence, we have ${\sf M}_p(C_{2q+1}) \rightarrow {\sf M}_m(C_{2q+1})$.

\medskip

{\bf Stage 2}.\
Define a mapping $\tau: V({\sf M}_m(C_{2q+1})) \rightarrow
V({\sf M}_m(C_{2s+1}))$ as follows. Let $\tau(u)=u$. When
$0 \leqslant i \leqslant m$, let
$$
\tau(\la i, j \ra )=  \left \{
\begin{array}{ll}
\la i, j \ra    & \mbox{ if $0 \leqslant j \leqslant 2s$}, \\
\la i, 2s \ra   & \mbox{ if $2s+2 \leqslant j \leqslant 2q$
 and $j$ is even}, \\
\la i, 2s-1 \ra & \mbox{ if $2s+1 \leqslant j \leqslant 2q-1$
and $j$ is odd}. \\
\end{array}
\right.
$$

To check that $\tau$ is a homomorphism from ${\sf M}_m(C_{2q+1})$
to ${\sf M}_m(C_{2s+1})$, let $xy$ be an edge of ${\sf
M}_m(C_{2q+1})$ having the form $x=\la i,j \ra$ and $y=\la k,j+1
\ra$, where $i,k \in \{0, 1, \ldots , m\}$, $k = i \pm 1$ if $i >
0$, and $k = 0$ if $i = 0$. $$
\tau(x)\tau(y)= \left \{
\begin{array}{ll}
\la i,j\ra \la k,j+1 \ra  &  \mbox{ if $0 \leqslant j \leqslant 2s-1$}, \\
\la i,2s\ra \la k,2s-1 \ra    &  \mbox{ if $2s \leqslant j \leqslant 2q-2$
and $j$ is even}, \\
\la i,2s-1 \ra \la k,2s \ra   &   \mbox{ if $2s+1 \leqslant j \leqslant 2q-1$
and $j$ is odd}, \\
\la i,2s\ra \la k,0 \ra    &  \mbox{ if $j=2q$}.\\
\end{array}
\right.
$$

Next if $x=\la m,j \ra$, $y=u$, and $0 \leqslant j \leqslant 2q$,
then $\tau(x)\tau(y)=\la m,h \ra u$ for an appropriate
$h \in \{0, 1, \dots , 2s \}$.

We see that $\tau(x)\tau(y) \in E({\sf M}_m(C_{2s+1}))$ in all cases.
Hence, we have ${\sf M}_m(C_{2q+1}) \rightarrow {\sf M}_m(C_{2s+1})$.
\end{proof}

\begin{theorem}\label{product}
If $m, n, s, t$ are all positive integers, then the direct product
${\sf M}_m(C_{2s+1})\times {\sf M}_n(C_{2t+1})$ is not a cover graph.
\end{theorem}

\begin{proof}
Assume $p=\max\{m,n\}$ and $q=\max\{s, t\}$. By Lemma
\ref{homo-to-subgraph} and Lemma \ref{big-homo-to-small}, we have
${\sf M}_p(C_{2q+1})\times {\sf M}_p(C_{2q+1}) \rightarrow
{\sf M}_m(C_{2s+1})\times {\sf M}_n(C_{2t+1})$. Then Lemma \ref{2homoto1},
Corollary \ref{G*G}, and Theorem \ref{lai} finish the proof.
\end{proof}

\begin{theorem}\label{product-of-nonbipartite}
The direct product ${\sf M}_m(G)\times
{\sf M}_n(H)$ is a cover graph if and only if $G$
or $H$ is bipartite.
\end{theorem}

\begin{proof}
Without loss of generality, we may suppose that $G$ is bipartite.
Then ${\sf M}_m(G)$ is a cover graph by
Theorem \ref{lai}. Since ${\sf M}_m(G)\times {\sf M}_n(H)
\rightarrow {\sf M}_m(G)$, ${\sf M}_m(G)\times {\sf M}_n(H)$
is a cover graph by Lemma \ref{2homoto1}.

Conversely, if both $G$ and $H$ are not bipartite, then we
may assume that $C_{2s+1}$ is a subgraph of $G$ and $C_{2t+1}$
is a subgraph of $H$ for some positive integers
$s$ and $t$. Hence ${\sf M}_m(C_{2s+1})$ is a subgraph of
${\sf M}_m(G)$ and ${\sf M}_n(C_{2t+1})$ is a subgraph of
${\sf M}_n(H)$. Then ${\sf M}_m(C_{2s+1})\times {\sf M}_n(C_{2t+1})$
is a subgraph of ${\sf M}_m(G)\times {\sf M}_n(H)$ by Lemma
\ref{product-of-subgraph}. Since ${\sf M}_m(C_{2s+1})\times
{\sf M}_n(C_{2t+1})$ is not a cover graph, neither is ${\sf M}_m(G)
\times {\sf M}_n(H)$ a cover graph.
\end{proof}

\bigskip

%
\section{A quick application}
%

To conclude this paper, we give a quick application
of our main results.
In recent years, there has been an intensive investigation
into the circular chromatic number $\chi_c(G)$ of a graph $G$.
Since $\chi(G)-1 < \chi_c(G) \leqslant \chi(G)$,
the circular chromatic number is regarded as a refinement of the
ordinary chromatic number. Zhu \cite{zhu} and \cite{zhu2}
provide comprehensive surveys on the circular chromatic
number. The reader is referred to them for basic notions
and results. Among a number of approaches to define
the circular chromatic number, one is via the minimum
imbalance of acyclic orientations.

Let $D$ be an acyclic orientation of the graph $G$. For an
undirected cycle $C$ of $G$, we choose one of the two
traversals of $C$ as the positive direction. An arc is
said to be {\em forward} if its orientation under $D$ is
along the positive direction of $C$, otherwise it is said
to be {\em backward}. We use $(C,D)^+$ (or $(C,D)^-$) to
denote the set of all forward (or backward) arcs of $C$
with respect to $D$. The {\em imbalance} Imb$(D)$ of $D$ is
defined to be
$$
\max \left\{ \max \left\{\frac{|C|}{|(C,D)^+|}, \frac{|C|}{|(C,D)^-|}\right\}
 \  \bigg{|} \  \mbox{ $C$ is a cycle of $G$} \right\}.
$$
By convention, Imb$(D)=2$ if $G$ has no cycles.
A result of Goddyn et al. \cite{gtz} implies that
$$
\chi_c(G)=\min \{ \mbox{Imb}(D) \mid \mbox{$D$ is an acyclic
orientation of $G$}\}.
$$

We can generalize Theorem \ref{chig} in the context of
circular chromatic number as follows.

\begin{theorem}\label{xc}
A graph $G$ is a cover graph if $\chi_c(G)< g(G)$.
\end{theorem}

\begin{proof}
There exists some acyclic orientation $D_0$ of $G$ with
$\chi_c(G) = \mbox{Imb}(D_0)$ by the result of Goddyn et al.
Suppose that $D_0$ has a dependent arc. Then the reversal of
that dependent arc will create a directed cycle with an
underlying cycle $C_0$ of $G$. Hence $\max\{|C_0|/|(C_0,D_0)^+|$,
$|C_0|/|(C_0,D_0)^-| \}$ $ \geqslant |C_0|$ $ \geqslant g(G)$. It follows that
$\chi_c(G) = \mbox{Imb}(D_0) \geqslant g(G)$ which contradicts our assumption.
Therefore $d_{\min}(G)=0$ and $G$ is a cover graph.
\end{proof}

\bigskip

It follows from Theorems \ref{lai} and \ref{xc} that, $\chi_c({\sf
M}_m(C_{2s+1}))=4$ for $s > 1$ since $\chi({\sf
M}_m(C_{2s+1}))=g({\sf M}_m(C_{2s+1}))=4$. Now let $s, t > 1$,
and let $G={\sf M}_m(C_{2s+1}) \times {\sf M}_n(C_{2t+1})$. Since
the girth of $G$ is 4 and $G$ is not a cover graph, we have
$\chi_c(G) \geqslant 4$ by Theorem \ref{xc}. On the other hand,
$G \rightarrow {\sf M}_m(C_{2s+1})$ implies that $\chi_c(G)
\leqslant \chi(G) \leqslant
\chi({\sf M}_m(C_{2s+1}))=4$.

\begin{corollary}
Let $s, t > 1$. Then
$$
\chi_c({\sf M}_m(C_{2s+1}) \times {\sf M}_n(C_{2t+1})) =
\min \{\chi_c({\sf M}_m(C_{2s+1})), \chi_c({\sf M}_n (C_{2t+1}))\}.
$$
\end{corollary}

Note that the above Corollary also follows directly from the
main argument of the proof of Theorem \ref{product}. It is also an immediate
consequence of the following
much stronger result established in Tardif
\cite{tardif}: $\chi_c(G\times H) =$
$\min\{\chi_c(G),$
$\chi_c(H)\}$ if $\min\{\chi_c(G),$
$\chi_c(H)\} \leqslant 4$.

\bigskip

{\bf Acknowledgement.}\  the authors thank the referees for helping them to
improve the presentation of this paper.


\end{document}